\documentclass{article}
\usepackage{amsfonts,enumerate}
\usepackage{amsthm,fullpage}
\usepackage{amsmath,mathdots}
\usepackage{amssymb,mathrsfs}
\newcommand{\al}{\alpha}
\newcommand{\si}{\sigma}
\newcommand{\Ckh}{\mathcal{C}_{kh}}
\newcommand{\Wkh}{\mathcal{W}_{kh}}
\newcommand{\Rkh}{\mathcal{R}_{kh}}
\newcommand{\la}{\lambda}
\newcommand{\R}{\mathbb{R}}
\newcommand{\p}{\partial}
\newcommand{\scr}{\mathscr}
\newcommand{\eps}{\varepsilon}
\renewcommand{\cal}{\mathcal}
\DeclareMathOperator{\im}{\mathrm{im}}
\DeclareMathOperator{\re}{\mathrm{Re}}
\DeclareMathOperator{\imag}{\mathrm{Im}}
\DeclareMathOperator{\codim}{\mathrm{codim}}
\DeclareMathOperator{\spn}{\mathrm{span}}
\newtheoremstyle{NTS1}{\topsep}{\topsep}%
     {\it}
     {}
     {\bf}
     {}
     { }
     {\thmname{#1}\thmnumber{ #2}\thmnote{ (#3)}}

\theoremstyle{NTS1}
\newtheorem{thm}{Theorem}[section]

\theoremstyle{remark}
\newtheorem*{prems}{Remarks}

\theoremstyle{remark}
\newtheorem*{rem}{Remark}
\numberwithin{equation}{section}
\usepackage[normalem]{ulem}

\begin{document}

\title{Global bifurcation for steady finite-depth capillary-gravity waves with constant vorticity}
\author{Peter de Boeck}
\date{}
\maketitle
\begin{abstract}
This paper is concerned with two-dimensional, steady, periodic water waves propagating at the free surface of water in a flow of constant vorticity over an impermeable flat bed. The motion of these waves is assumed to be governed both by surface tension and gravitational forces. A new reformulation of this problem is given, which is valid without restriction on the geometry or amplitude of the wave profiles and possesses a structure amenable to global bifurcation. The existence of global curves and continua of solutions bifurcating at certain parameter values from flat, laminar flows is then established through the application of real-analytic global bifurcation in the spirit of Dancer and the degree-theoretic global bifurcation theorem of Rabinowitz, respectively.\end{abstract}
\section{Introduction}
For many years, the mathematical study of water waves was confined to the irrotational case, that is, of flows without vorticity (cf., e.g., \cite{G,JFT}). The main mathematical advantage to the study of irrotational flows is that one may use conformal mappings to reformulate the problem on a fixed domain, and this is the standard approach in the theory of irrotational water waves. Although considering flows without vorticity provides a certain mathematical simplification and has a real-world interpretation as waves travelling over a uniform current, it is well-known that non-uniform currents lead to flows with non-zero vorticity \cite{Cbook}. The example of a constant vorticity, as treated in this work, is a distinctive attribute of tidal currents, as discussed in \cite{Cbook,TdSP}.

In more recent years, there has been a surge in papers considering the rotational case, principally initiated by the breakthrough in \cite{CS}, that address symmetry \cite{CEW,CE04,E}, regularity \cite{CE11,CV} and existence \cite{CS,CV,V}. However, the majority of the literature is focused on the pure-gravity wave scenario (that is, where the governing force is gravitational in nature and the effects of surface tension are neglected) and it is only comparatively recently that the pure-capillary (where gravitational forces are neglected) and capillary-gravity (considering both gravitational and surface tension forces) cases have been studied in great detail. This has yielded results concerning existence of both small-amplitude \cite{HM,M,Mcap,Wcap,Wc-g} and large-amplitude \cite{Mgb,Walsh} waves of both pure-capillary and capillary-gravity type, as well as other interesting results \cite{Mreg,MM}.

However, while great strides have been made in this area, there has yet to be a work where there is no limitation imposed on the amplitude or geometry of the wave profiles. Indeed, where small-amplitude existence is studied it is not possible to draw any conclusions of large-amplitude existence. Equally, in both \cite{Mgb} and \cite{Walsh} the authors use an approach that requires the profile necessarily to be a graph, viz. a ``flattening'' transformation and a ``semi-Lagrangian" hodograph transformation, respectively. But waves possessing an overhanging profile is a circumstance that is both possible and expected in the case of constant-vorticity capillary-gravity waves. Indeed, numerical studies \cite{KS1,KS2,TdSP} have shown that the addition of constant vorticity to the pure-gravity wave problem generates overhanging profiles and similarly when the effects of surface tension are included (but vorticity is neglected) \cite{AAW,OS}.

Focusing specifically on those works that consider constant (non-zero) vorticity \cite{CSV,CV,M,Mcap,Mreg,MM}, as in this work, the common approach has been to adapt the method of conformal mappings from the irrotational setting. As these many works demonstrate, this reformulation on a fixed domain is an extension of the corresponding irrotational formulation, with extra terms appearing exactly due to the influence of vorticity. However, this reformulation does not seem to be obviously well-suited to the application of the theory of global bifurcation and, as remarked above, when treating the case of global existence, most authors have employed an alternative formulation.

By contrast, in this paper, we present a new formulation that is related to the reformulation via conformal mappings. This new formulation possesses the compactness structure that allows us to prove the existence of global families of finite-depth capillary-gravity water waves with non-zero constant vorticity. The main novelty of this work is that no mathematical restriction is placed upon the profile of the wave; in particular, our analysis provides for profiles that are overhanging and of arbitrary amplitude.

The plan of this paper is as follows.  In \S2 we present the problem and its various reformulations. In \S3 we present the results of global bifurcation.
\section{The free boundary problem}
For any $L>0$, we say that a domain $\Omega\subset\R^2$ is an $L$-periodic strip-like domain if it is horizontally unbounded, having a boundary which consists of the real axis, denoted by $\cal{B}$, and a (possibly unknown) curve $\cal{S}$, given in parametric form by \begin{subequations}\begin{equation}\cal{S}=\{(u(s),v(s)):s\in\R\}\label{surface}\end{equation}such that \begin{equation}u(s+L)=u(s)+L,\quad v(s+L)=v(s)\text{ for all }s\in\R.\label{fbper}\end{equation}\end{subequations}Then, given $L>0$, the problem of finding $L$-periodic travelling gravity-capillary water waves with constant vorticity $\gamma$ in a flow of finite depth over a flat bed can be formulated as the free boundary problem of finding an $L$-periodic strip-like domain $\Omega\subset\R^2$ (where the curve $\cal{S}$ representing the free surface of the water is \emph{a priori} unknown) and a function $\psi:\Omega\to\R$, satisfying the following equations and boundary conditions:\begin{subequations}\label{form}\begin{align}\Delta\psi&=-\gamma\quad\,\,\text{in }\Omega,&\\\psi&=-m\quad\text{on }\cal{B},\\\psi&=0\,\qquad\text{on }\cal{S},\\|\nabla\psi|^2+2gv-2\si\frac{u_sv_{ss}-u_{ss}v_s}{(u_s^2+v_s^2)^{3/2}}&=Q\quad\;\;\;\text{on }\cal{S}.\end{align}\end{subequations}The function $\psi$ represents the stream function, giving the velocity field $(\psi_Y,-\psi_X)$ in a frame moving at the constant wave speed; the parameterisation in \eqref{surface} has also been chosen so as to be time-independent in this moving frame. The constants $g,m\in\R,\si>0$ in \eqref{form} are, respectively, the constant of gravitational acceleration, the relative mass flux and the coefficient of surface tension. The constant $Q\in\R$ is related to the hydraulic head.

These equations may be derived directly from the Euler equations for incompressible fluid flow, as in \cite{M}.

For $p\ge0$ and $\al\in(0,1)$, let $C^{p,\al}$ denote the space of functions whose partial derivatives of order up to $p$ are H\"older continuous of exponent $\al$ over their domain of definition. We denote by $C^{p,\al}_L$ the subspace of $C^{p,\al}(\R)$ consisting of functions that are $L$-periodic. Taking $L=2\pi$, we may define the mean over one period of a function $f\in C^{p,\al}_{2\pi}$ by $$[f]:=\frac1{2\pi}\int_{-\pi}^\pi f(t)\,\mathrm{d}t.$$We denote by $C^{p,\al}_{2\pi,0}$ those functions $f\in C^{p,\al}_{2\pi}$ with zero mean, i.e. $[f]=0$.

Throughout this paper, we are interested in solutions $(\Omega,\psi)$ of the problem \eqref{form} of class $C^{2,\al}$, for some $\al\in(0,1)$, in the sense that $\cal{S}$ has a parameterisation \eqref{surface} with $(s\mapsto u(s)-s),v$ both functions in $C^{2,\al}_L$ with\begin{equation}u'(s)^2+v'(s)^2\ne0\quad\text{for all }s\in\R,\label{nostag}\end{equation}while $\psi\in C^\infty(\Omega)\cap C^{1,\al}(\overline{\Omega})$.
\subsection{A first reformulation}
Using the approach presented in \cite{CV}, it was shown in \cite{M} that \eqref{form} can be reformulated as a problem on a fixed domain, via conformal mappings. The approach to this formulation is as follows.

For and $d>0$ let $$\cal{R}_d:=\{(x,y)\in\R^2:-d<y<0\}.$$For any $w\in C^{p,\al}_{2\pi}$ let $W\in C^{p,\al}(\overline{\cal{R}_d})$ be the unique solution of \begin{align}\Delta W&=0\quad\text{in }\cal{R}_d,\nonumber\\W(x,-d)&=0,\quad x\in\R,\label{Wsystem}\\W(x,0)&=w(x),\quad x\in\R.\nonumber\end{align}The function $(x,y)\mapsto W(x,y)$ is $2\pi$-periodic in $x$ throughout $\cal{R}_d$. Let $Z$ be the unique (up to a constant) harmonic conjugate of $W$, that is, such that $Z+iW$ is a holomorphic function on $\cal{R}_d$. As described thoroughly in \cite[\S2]{CV}, if $w\in C^{p,\al}_{2\pi,0}$ then the function $(x,y)\mapsto Z(x,y)$ is $2\pi$-periodic in $x$ throughout $\cal{R}_d$. We may then prescribe the constant in the definition of $Z$ by requiring that the function $x\mapsto Z(x,0)$ has zero mean.

Then the operator $\cal{C}_d:C^{p,\al}_{2\pi,0}\to C^{p,\al}_{2\pi,0}$ is defined by $$\cal{C}_d(w)(x)=Z(x,0),\quad x\in\R.$$(It was shown in \cite{CV} that the operator $C_d$ satisfies a sort of Privalov's Theorem, and maps $C^{p,\al}$ functions into $C^{p,\al}$ indeed.)

In fact, it is possible to define the operator $\cal{C}_d$ for $2\pi$-periodic functions in $L^2$, following the argument in \cite{CV}. Let $L^2_{2\pi}$ denote the space of $2\pi$-periodic (locally) square-integrable functions of one real variable, and denote by $L^2_{2\pi,0}$ its subspace consisting of elements with zero mean over one period. Every function $w\in L^2_{2\pi}$ has a Fourier series expansion given by $$w(t)=[w]+\sum_{n=1}^\infty a_n\cos{nt}+\sum_{n=1}^\infty b_n\sin{nt}.$$

For any $d>0$ and $w\in L^2_{2\pi}$, the function $W:\cal{R}_d\to\R$, given by $$W(x,y)=\frac{[w]}{d}(y+d)+\sum_{n=1}^\infty a_n\frac{\sinh(n(y+d))}{\sinh(nd)}\cos{nx}+\sum_{n=1}^\infty b_n\frac{\sinh(n(y+d))}{\sinh(nd)}\sin{nx},$$is the unique solution of \eqref{Wsystem}, with the third condition being satisfied in the sense $$\lim_{y\to0}\|W(\cdot,y)-w\|=0,$$where $\|\cdot\|$ denotes the standard $L^2_{2\pi}$ norm or an equivalent one.

If $w\in L^2_{2\pi,0}$ then any harmonic function $Z$ in $\cal{R}_d$ such that $Z+iW$ is holomorphic is given by $$Z(x,y)=C+\sum_{n=1}^\infty a_n\frac{\cosh(n(y+d))}{\sinh(nd)}\sin{nx}-\sum_{n=1}^\infty b_n\frac{\cosh(n(y+d))}{\sinh(nd)}\cos{nx},$$where $C$ is a constant. Choosing $C=0$, we can define $\cal{C}_d(w)$ to be the unique function in $L^2_{2\pi,0}$ such that $$\lim_{y\to0}\|Z(\cdot,y)-\cal{C}_d(w)\|=0;$$that is, $$\cal{C}_dw(t)=\sum_{n=1}^\infty a_n\coth(nd)\sin{nt}-\sum_{n=1}^\infty b_n\coth(nd)\cos{nt}.$$

It was also shown in \cite{CV} that for any $L$-periodic strip-like domain, $\Omega$, there exists a unique positive number $h$, called its conformal mean depth, such that $\Omega$ is the image of a conformal mapping $U+iV$ from $\Rkh$ such that \begin{eqnarray*}U(x+2\pi,y)&=&U(x,y)+L\\V(x+2\pi,y)&=&V(x,y)\end{eqnarray*}for $(x,y)\in\Rkh$, where $L=2\pi/k$. In particular, it follows that $\cal{S}$ is parameterised by functions $u\in C^{2,\al},v\in C^{2,\al}_{2\pi}$ with mean $[v]=h$. It thus follows that for any $L$-periodic strip-like domain we may make a vertical translation to place its flat bottom at $Y=-h$, where $h$ is its conformal mean depth, whence it is clear that the free surface will admit a parameterisation with a zero-mean $Y$ coordinate.

Bearing in mind that we are free to make this vertical translation we may now present the reformulation, as given in \cite[Theorem 1]{M}, but where suitable adjustments have been made to account for this vertical translation, and also a substitution has been made to rescale the function of interest.

We have that the problem of finding $(\Omega,\psi)$ of class $C^{2,\al}$ solving \eqref{form} is equivalent to finding an $h>0$ and $w\in C^{2,\al}_{2\pi,0}$ such that \begin{subequations}\label{next}\begin{align}&\left\{\frac{m}{h}-\frac{\gamma h}{2}+\frac\gamma k\left(\frac{[w^2]}{2kh}+\Ckh(ww')-w-w\Ckh w'\right)\right\}^2\nonumber\\&\qquad\qquad\qquad\qquad\qquad=\left(Q-\frac{2gw}k+2\sigma k\frac{(1+\Ckh w')w''-w'\Ckh w''}{\Wkh(w)^{3/2}}\right)\Wkh(w)\label{2eqn}\\&w(t)>-kh\quad\text{for all }t\in\mathbb{R}&\label{abovebed}\\&\text{the mapping }t\mapsto\left(t+\Ckh w(t),w(t)\right)\text{ is injective on }\mathbb{R}\label{injcond}\\&\Wkh(w)(t)\ne0\quad\text{for all }t\in\mathbb{R},\label{Wkh0}\end{align}\end{subequations}where \begin{equation}\Wkh(w)=w'^2+(1+\Ckh w')^2\label{Wkh}\end{equation}and, as above, $L=2\pi/k$.

Here it should be noted that the formulation in \cite{M} has been altered, using the general theory of periodic harmonic functions on a strip from \cite{CV}, to use the operator $\Ckh$, the periodic Hilbert transform for a strip.  A full and coherent treatment of the properties of this and related operators can be found in \cite{CV}.

Given an $h>0$ and $w\in C^{2,\al}_{2\pi,0}$ satisfying \eqref{next}, it follows that the fluid region is bounded by the curves $$\mathcal{B}=\{(X,-h):X\in\mathbb{R}\},\qquad\mathcal{S}=\left\{\left(\frac1k(t+\Ckh w(t)),\frac{w(t)}k\right):t\in\mathbb{R}\right\}.$$Indeed, suppose that $h,w$ satisfy \eqref{next}. Letting $V$ be the harmonic function on $\Rkh$ satisfying $$V(x,-kh)=-h\quad\text{and}\quad V(x,0)=w(x)/k,$$and $U$ such that $U+iV$ is holomorphic, then it follows that $U+iV$ is a conformal mapping from $\Rkh$ onto the domain $\Omega$, which is bounded by $\cal{B}$ and $\cal{S}$, that extends as a homeomorphism between the closures of these domains. The conditions satisfied by $w$ then imply that $\Omega$ is a $2\pi/k$-periodic strip-like domain with conformal mean depth $h$. Moreover, if $\zeta\in C^{2,\al}(\overline{\Rkh})\cap C^{\infty}(\Rkh)$ is the unique solution of \begin{align*}\Delta\zeta&=0\quad\text{in }\Rkh,\\\zeta(x,-kh)&=0\quad x\in\R,\\\zeta(x,0)&=m+\frac{\gamma}{2k^2}w^2(x)\quad x\in\R,\end{align*}and we define $\psi$ by $$\psi(U(x,y),V(x,y))=\zeta(x,y)-m-\frac{\gamma}{2}V^2(x,y)\qquad (x,y)\in\Rkh,$$then \eqref{next} implies that $\psi$ satisfies \eqref{form}.

We remark at this point that it should be noted that for the purposes of the analysis we shall ignore the conditions \eqref{abovebed}, \eqref{injcond}, simply solving \eqref{2eqn} and \eqref{Wkh0} then throwing away any solutions which violate \eqref{abovebed}, \eqref{injcond} afterwards. Hence, throughout we shall remove this conditions, leaving it to the reader to remember that any solutions found may in fact have to be ignored if it happens that they violate one or both of these conditions.
\subsection{An additional reformulation of the problem}
As was shown earlier, and in view of the previous remark, we aim to find $h,k,m,Q,\gamma\in\mathbb{R},\;h,k>0$ and $w\in C^{2,\al}_{2\pi,0}$ such that \eqref{2eqn} holds, subject to \eqref{Wkh0}. It is convenient to rewrite \eqref{2eqn} in the form\begin{multline}\left\{\frac{m}{h}-\frac{\gamma h}{2}+\frac\gamma k\left(\frac{[w^2]}{2kh}+\Ckh(ww')-w-w\Ckh w'\right)\right\}^2\Wkh(w)^{-1/2}\\=\left(Q-\frac{2gw}{k}\right) \Wkh(w)^{1/2}+2\sigma k\frac{(1+\Ckh w')w''-w'\Ckh w''}{\Wkh(w)}\label{eqn2a}.\end{multline}However, it is clear that this equation will not be suitable when we come to approach Global Bifurcation. For this we will require an alternative equation, the derivation of which is included in the following result on equivalence.\begin{thm} For $\si\ne0$, and for any $\gamma\in\mathbb{R}$ and $h,k>0$, we have that $m,Q\in\mathbb{R}$ and $w\in C^{2,\al}_{2\pi,0}$ with $$\Wkh(w)(t)\ne0\quad\text{for all }t\in\mathbb{R}$$satisfies \eqref{eqn2a} if and only if \begin{multline}Q=\left[\Wkh(w)^{1/2}\right]^{-1}\left[\left\{\frac{m}{h}-\frac{\gamma h}{k}+\frac{\gamma}{k}\left(\frac{[w^2]}{2kh}+\mathcal{C}_{kh}(ww')-w-w\mathcal{C}_{kh}w'\right)\right\}^2\Wkh(w)^{-1/2}\right.
\\\left.+\frac{2gw}k\Wkh(w)^{1/2}\right]\label{mudef}\end{multline}and\begin{multline}w''=\frac{w'}{2\sigma k}\Ckh\left(\left\{\frac{m}{h}-\frac{\gamma h}2+\frac{\gamma}k\left(\frac{[w^2]}{2kh}+\mathcal{C}_{kh}(ww')-w-w\mathcal{C}_{kh}w'\right)\right\}^2\Wkh(w)^{-1/2}-\right.\\\left.-\left(Q-\frac{2gw}k\right)\Wkh(w)^{1/2}\right)+\\+\frac{1}{2\sigma k}(1+\Ckh w')\left(\left\{\frac{m}{h}-\frac{\gamma h}2+\frac{\gamma}k\left(\frac{[w^2]}{2kh}+\mathcal{C}_{kh}(ww')-w-w\mathcal{C}_{kh}w'\right)\right\}^2\Wkh(w)^{-1/2}\right.-
\\\left.-\left(Q-\frac{2gw}k\right)\Wkh(w)^{1/2}\right)\label{2nd}.\end{multline}\end{thm}\begin{proof}We start in a more general setting that, whilst seemingly unrelated to the problem at hand, will turn out to be applicable when considering the equivalence in which we are interested.

For any $w\in C^{2,\alpha}_{2\pi}$, let us define a function $\hat w\in C^{0,\alpha}_{2\pi}$ by \begin{equation}\hat w(t):=\frac{(1+\Ckh w'(t))w''(t)-w'(t)\Ckh w''(t)}{\Wkh(w)(t)}.\label{vdef}\end{equation}Let $A$ be the holomorphic function on $\Rkh$, $2\pi$-periodic in $x$, real-valued on $y=-kh$, with values on the real axis  given by $\Ckh w''+iw''$, and let $B$ be the holomorphic function on $\Rkh$, $2\pi$-periodic in $x$, real-valued on $y=-kh$, with values on the real axis  given by $(1+\Ckh w')+iw'$. Let $C$ be given by $C:=A/B$. Then $C$ is also holomorphic on $\Rkh$, $2\pi$-periodic in $x$, real-valued on $y=-kh$, with values on the real axis given by $$J(t)=\frac{\mathcal{C}_{kh}w''(t)+iw''(t)}{(1+\mathcal{C}_{kh}w'(t))+iw'(t)}.$$ It can then be seen that $\hat w(t)=\imag C(t)$, which implies that $[\hat w]=0$ and $\re C(t)=\Ckh\hat w(t)+K$ for some constant $K$. Indeed, it was shown in \cite{CV} that for a function $Z+iW$ holomorphic on $\Rkh$, real-valued on $y=-kh$, with $W(x,0)=w(x)$, the function $Z$ is $2\pi$-periodic if and only if $[w]=0$. Since $\re C$ is $2\pi$-periodic, it then follows that $[\hat w]=0$, whence it makes sense to define $\Ckh\hat w$ and we conclude that $\re C(t)=\Ckh\hat w(t)+[\re C]$.

But now, since $$\re C(t)=\frac{(1+\Ckh w'(t))\Ckh w''(t)+w'(t)w''(t)}{\Wkh(w)(t)}=\frac{\mathrm{d}}{\mathrm{d}t}\log\left(\Wkh(w)(t)\right)^{1/2},$$it follows that $K=0$, and hence \begin{equation}\Ckh \hat w=\frac{\Ckh w''(1+\Ckh w')+w'w''}{\Wkh(w)}.\label{Cvdef}\end{equation}After some simple algebra, we can then deduce, from \eqref{vdef} and \eqref{Cvdef}, that \begin{equation}w''=w'\Ckh\hat w+(1+\Ckh w')\hat w.\label{w''v}\end{equation}Suppose now that $(m,Q,w)$ is a solution of \eqref{eqn2a}. Note that \eqref{eqn2a} may be written in the form$$\left\{\frac{m}{h}-\frac{\gamma h}{2}+\frac\gamma k\left(\frac{[w^2]}{2kh}+\Ckh(ww')-w-w\Ckh w'\right)\right\}^2\Wkh(w)^{-1/2}=\left(Q-\frac{2gw}k\right)\Wkh(w)^{1/2}+2\sigma k\hat w.$$This then implies, by taking averages and using the fact that $[\hat w]=0$, that\begin{multline}\left[\left\{\frac{m}{h}-\frac{\gamma h}{k}+\frac{\gamma}{k}\left(\frac{[w^2]}{2kh}+\mathcal{C}_{kh}(ww')-w-w\mathcal{C}_{kh}w'\right)\right\}^2\Wkh(w)^{-1/2}\right.\\-\left.\left(Q-\frac{2gw}k\right)^{\phantom{2}}\!\!\!\Wkh(w)^{1/2}\right]=0,\label{vmean0}\end{multline}which yields that $Q$ is given by (\ref{mudef}). Moreover, \eqref{eqn2a} may also be written as \begin{multline}\hat w=\frac{1}{2\sigma k}\left\{\frac{m}{h}-\frac{\gamma h}{2}+\frac{\gamma}k\left(\frac{[w^2]}{2kh}+\mathcal{C}_{kh}(ww')-w-w\mathcal{C}_{kh}w'\right)\right\}^2\Wkh(w)^{-1/2}-\\-\frac{1}{2\sigma k}\left(Q-\frac{2gw}k\right)\Wkh(w)^{1/2}\label{veqn},\end{multline}with $Q$ given by \eqref{mudef}. From this we obtain that\begin{multline}\Ckh \hat w=\Ckh\left(\frac{1}{2\sigma k}\left\{\frac{m}{h}-\frac{\gamma h}{2}+\frac{\gamma}k\left(\frac{[w^2]}{2kh}+\mathcal{C}_{kh}(ww')-w-w\mathcal{C}_{kh}w'\right)\right\}^2\Wkh(w)^{-1/2}-\right.\\\left.-\frac{1}{2\sigma k}\left(Q-\frac{2gw}k\right)\Wkh(w)^{1/2}\right),\label{Cveqn}\end{multline}and the validity of \eqref{2nd} is then a consequence of \eqref{w''v}, \eqref{veqn}, and \eqref{Cveqn}.

Suppose now that $(m,Q,w)$ are such that \eqref{mudef} and \eqref{2nd} are satisfied. Let $\tilde w$ be given by \begin{multline}\tilde w=\frac{1}{2\sigma k}\left\{\frac{m}{h}-\frac{\gamma h}{2}+\frac{\gamma}k\left(\frac{[w^2]}{2kh}+\mathcal{C}_{kh}(ww')-w-w\mathcal{C}_{kh}w'\right)\right\}^2\Wkh(w)^{-1/2}-\\-\frac{1}{2\sigma k}\left(Q-\frac{2gw}k\right)\Wkh(w)^{1/2}\label{tiveqn},\end{multline}with $Q$ given by \eqref{mudef}. Then $[\tilde w]=0$ and \eqref{2nd} takes the form \begin{equation}w''=w'\Ckh\tilde w+(1+\Ckh w')\tilde w.\label{tiw''v}\end{equation}Let $D$ be the holomorphic function on $\Rkh$, $2\pi$-periodic in $x$, real-valued on $y=-kh$, with values on the real axis  given by $\Ckh\tilde w+i\tilde w$, and let $E:=BD$, where $B$ is as defined earlier. Then $E$ is also holomorphic on $\Rkh$, $2\pi$-periodic in $x$, real-valued on $y=-kh$, and its values on the real axis are given by$$E(t)=((1+\Ckh w'(t))+iw'(t))(\Ckh\tilde w(t)+i\tilde w(t)).$$Then$$\re E(t)=(1+\Ckh w'(t))\Ckh\tilde w(t)-w'(t)\tilde w(t),\qquad\imag E(t)=w'(t)\Ckh\tilde w(t)+(1+\Ckh w'(t))\tilde w(t).$$It now follows from \eqref{tiw''v} that there exists a constant $\tilde K$ such that $$\tilde K+\Ckh w''=(1+\Ckh w')\Ckh\tilde w-w'\tilde w.$$Multiplying this equation by $w'$ and multiplying \eqref{tiw''v} by $1+\Ckh w'$ yields \begin{align}w'(\tilde K+\Ckh w'')&=w'(1+\Ckh w')\Ckh\tilde w-w'^2\tilde w,\label{w'2v}\\(1+\Ckh w')w''&=w'(1+\Ckh w')\Ckh\tilde w+(1+\Ckh w')^2\tilde w\label{ckh2v}\end{align}respectively. It follows from \eqref{w'2v} and \eqref{ckh2v} that\begin{equation}\tilde w=\frac{(1+\Ckh w')w''-w'(\tilde K+\Ckh w'')}{\Wkh(w)}.\label{tildew}\end{equation}By an entirely similar argument, it can also be shown that\begin{equation}\Ckh\tilde w=\frac{w'w''+(1+\Ckh w')(\tilde K+\Ckh w'')}{\Wkh(w)}.\label{Ckhtildew}\end{equation}However, since $\Ckh\tilde w$ must have zero mean, it follows that $$\tilde K\left[\frac{1+\Ckh w'}{\Wkh(w)}\right]=0;$$this because, from \eqref{Ckhtildew}, $$\Ckh\tilde w=\frac{w'w''+(1+\Ckh w')\Ckh w''}{\Wkh(w)}+\tilde K\frac{1+\Ckh w'}{\Wkh(w)}$$and we saw earlier that the first term on the right-hand side has zero mean. However, we can show that$$\left[\frac{1+\Ckh w'}{\Wkh(w)}\right]\ne0,$$whence it follows that necessarily $\tilde K=0$. Indeed, consider the holomorphic function $B$, as above. From the general remarks in \S2.1, it is clear that $[\re B(\cdot,y)]$ is constant as a function of $y$ (were it not, it would follow that $\imag B$ is not $2\pi$-periodic in $x$, which is false). Therefore, from its definition, $$[\re B(\cdot,-kh)]=[\re B(\cdot,0)]=1$$and, since we have $B(x,-kh)\in\R$ and require $B\ne0$ on $\overline{\Rkh}$, we obtain that $B(x,-kh)>0$. Then it is trivial to observe that $(1/B)(x,-kh)$ is real-valued and strictly positive, yielding $$\left[\frac1B(\cdot,-kh)\right]>0.$$By an entirely similar argument, one concludes that $$\left[\re\frac1B(\cdot,0)\right]=\left[\re\frac1B(\cdot,-kh)\right];$$observing that the left-hand side is equal to $[(1+\Ckh w')/\Wkh(w)]$ and the right-hand side is equal to $[(1/B)(\cdot,-kh)]$, the claim follows and $\tilde K=0$ indeed. 

In light of \eqref{tiveqn}, it is thus trivial to observe that \eqref{eqn2a} arises immediately from \eqref{tildew}, and the proof is complete.\end{proof}
\section{Global Bifurcation}\label{GB}Let $C^{p,\al}_{2\pi,0,e}$ denote the subspace of $C^{p,\al}_{2\pi,0}$ consisting of even functions.  We are attempting to find solutions $(m,w)$ to \eqref{2nd} in the space $\R\times C^{2,\al}_{2\pi,0,e}$. We first note that $(m,0)$ is a solution for any $m\in\R$. Indeed, it is easy to check -- with the definition for $Q$ given by \eqref{mudef} --  that \eqref{2nd} is also satisfied for any choice of $m\in\R$ when $w\equiv0$. These trivial solutions correspond to laminar flows with a flat free surface.

Let $X=C^{2,\alpha}_{2\pi,0,e}$ be the space of twice H\"older-continuously differentiable, $2\pi$-periodic, even functions of zero mean. Let $\mathcal{U}\subset X$ be the open subset given by $$\mathcal{U}=\{w\in X:w'^2+(1+\Ckh w')^2>0\}.$$
Also, denote by $D$ the differentiation operator $$D:C^{p,\alpha}_{2\pi,0}\to C^{p-1,\alpha}_{2\pi,0}\qquad\text{for }p\ge1,$$ given by $Dw=w'$. In fact, it can be seen that the zero mean condition makes $D$ an invertible operator; we denote its inverse by $D^{-1}$. Notice further that $D^{2}:=D\circ D$ maps even functions to even functions; we denote its inverse by $D^{-2}$.

It is convenient to make the change of parameter given by \begin{equation}\label{mtola}\la=\frac{m}{h}-\frac{h\gamma}2,\end{equation}which is seen to be bijective, and to define an operator $$K:\mathbb{R}\times\mathcal{U}\to C^{1,\alpha}_{2\pi,0,e}$$ via the right-hand side of equation \eqref{2nd}, where \eqref{mudef} and \eqref{mtola} have been taken into account, i.e. \begin{multline}K(\lambda,w)=\frac{w'}{2\sigma k}\Ckh\!\!\left(\left\{\la+\frac{\gamma}{k}\left(\frac{[w^2]}{2kh}+\mathcal{C}_{kh}(ww')-w-w\mathcal{C}_{kh}w'\right)\right\}^2\!\!\Wkh(w)^{-1/2}-\right.\\\left.-\left(Q(h\la+h^2\gamma/2,w)-\frac{2gw}k\right)\Wkh(w)^{1/2}\right)+\\+\frac{1}{2\sigma k}(1+\Ckh w')\!\!\left(\left\{\la+\frac{\gamma}{k}\left(\frac{[w^2]}{2kh}+\mathcal{C}_{kh}(ww')-w-w\mathcal{C}_{kh}w'\right)\right\}^2\!\!\Wkh(w)^{-1/2}\right.\\\left.-\left(Q(h\la+h^2\gamma/2,w)-\frac{2gw}k\right)\Wkh(w)^{1/2}\right).\label{Kdef}\end{multline} We thus see that \eqref{2nd} is reduced simply to $$D^2w=K(\lambda,w).$$
Moreover, by defining $f:=-D^{-2}\circ K$, we can now reformulate \eqref{2nd} -- and thus \eqref{form} -- in the following manner: \\\textit{Given $h,k>0$ and $g,\gamma,\si\in\mathbb{R}$, $\si\ne0$, find pairs $(\lambda,w)\in\mathbb{R}\times\mathcal{U}$ such that\begin{equation}w+f(\lambda,w)=0,\label{gbeqn}\end{equation}}

The standard approach in such problems is to apply the Rabinowitz Global Bifurcation Theorem \cite{R}. This theorem is based upon topological degree theory and has been refined by various authors; for a full account see, e.g., \cite{K}. However, when the mapping in question has a real-analytic structure, an alternative -- and arguably more refined -- theory of global bifurcation is available, first developed by Dancer \cite{D} and subsequently much studied and improved in \cite{BufTol}. When both theories are applicable, as it turns out to be the case in our problem, they yield related, but different conclusions.

We now state these theorems in a form which is suitable for our application. We present first the real-analytic theorem, in a form given in \cite{CSV}, where the authors noticed and corrected a slight error in \cite{BufTol}. Then we present the topological theorem, stated as in \cite{K}. This is followed by a brief explanation of some of the concepts mentioned there.

\begin{thm}Let \label{CSVthm}$X,Y$ be Banach spaces, $\cal{O}\subset\R\times X$ an open set and $F:\cal{O}\to Y$ be a real-analytic mapping. Suppose that \begin{enumerate}[a.]\item$(\la,0)\in\cal{O}$ and $F(\la,0)=0$ for all $\la\in\R$,\item for some $\la^*\in\R$ we have that $\dim\ker\p_xF[\la^*,0]=\codim\im\p_xF[\la^*,0]=1$ and, where we write $\ker\p_xF[\la^*,0]=\spn\{x^*\}$, that $\p^2_{\la x}F[\la^*,0](1,x^*)\notin\im\p_xF[\la^*,0]$,\item$\p_xF[\la,x]$ is a Fredholm operator of index $0$ whenever $F(\la,x)=0$ with $(\la,x)\in\cal{O}$,\item for some sequence $(\cal{Q}_j)_{j\in\mathbb{N}}$ of bounded closed subsets of $\cal{O}$ with $$\cal{O}=\bigcup_{j\in\mathbb{N}}\cal{Q}_j,$$ the set $\{(\la,x)\in\cal{O}:F(\la,x)=0\}\cap\cal{Q}_j$ is compact for each $j\in\mathbb{N}$. \end{enumerate}Then there exists in $\cal{O}$ a (unique up to reparameterisation) continuous curve $\scr{C}=\{(\la(s),x(s)):s\in\R\}$ of solutions to $F(\la,x)=0$ such that \begin{enumerate}\item $(\la(0),x(0))=(\la^*,0)$,\item$x(s)=sx^*+o(s)$ as $s\to0$ for $0<|s|<\eps$ for some $\eps>0$ sufficiently small,\item there is a neighbourhood of $(\la^*,0)$ in $\cal{O}$ and $\eps>0$ sufficiently small such that all solutions to $F(\la,x)=0$ with $x\ne0$ are given by $(\la(s),x(s))$ for $0<|s|<\eps$\item around each point of $\scr{C}$ there is a local real-analytic reparameterisation,\item one of the following alternatives occurs:\begin{enumerate}\item either, for every $j\in\mathbb{N}$ there exists $s_j>0$ such that $(\la(s),x(s))\notin\cal{Q}_j$ for all $|s|>s_j$, or\item there exists $T>0$ such that $(\la(s+T),x(s+T))=(\la(s),x(s))$ for all $s\in\R$.\end{enumerate}\end{enumerate}\end{thm}\begin{thm}Let\label{Rabthm} $X$ be a Banach space and $\mathcal{O}$ an open subset of $\mathbb{R}\times X$. Assume that $F\in C^1(\mathcal{O},X)$ is given by $F(\lambda,x)=x+f(\lambda,x)$, where $f:\mathcal{O}\to X$ is a completely continuous operator. Suppose further that $F(\lambda,0)=0$ for every $\lambda\in\mathbb{R}$ s.t. $(\lambda,0)\in\mathcal{O}$. Denote the closure of set of `non-trivial' solutions to $F(\lambda,x)=0$ by $\mathcal{S}:=\overline{\{(\lambda,x)\in\mathcal{O}:F(\lambda,x)=0,\,x\ne0\}}$. Assume that $\partial_xF[\lambda,0]$ has odd crossing number at $\lambda=\lambda_0$. Then $(\lambda_0,0)\in\mathcal{S}$ and, letting $\mathcal{C}$ be the connected component of $\mathcal{S}$ to which $(\lambda_0,0)$ belongs, we have the following alternatives: either \begin{enumerate}[$(i)$]\item $\mathcal{C}$ is unbounded in $\mathcal{O}$,\item $\mathcal{C}$ intersects the boundary of $\mathcal{O}$, or \item $\mathcal{C}$ contains some $(\lambda_1,0)\in\mathcal{O}$, with $\lambda_0\ne\lambda_1$.\end{enumerate}\end{thm}

We now give a brief explanation of what it means in this context for $\partial_xF[\lambda,0]$ to have odd crossing number at a particular value of the parameter, and refer the reader to \cite[\S II.3]{K} for a more comprehensive explanation.

In essence, the crossing number counts the number of eigenvalues near zero that cross from the left complex half-plane to the right as $\la$ passes through $\la_0$.  However, we can be more precise. Let $\la_0\in\R$ such that $\p_xF[\la_0,0]$ has $0$ as an isolated eigenvalue of finite (algebraic) multiplicity. By Leray-Schauder Theory, for $\la$ in a neighbourhood of $\la_0$  the eigenvalue $0$ of $\p_xF[\la_0,0]$ perturbs to some (finite) number of eigenvalues of $\p_xF[\la,0]$, each of which tends to $0$ as $\la\to\la_0$. The sum of the algebraic multiplicities of these eigenvalues is equal to the algebraic multiplicity of the eigenvalue $0$ of $\p_xF[\la_0,0]$. Considering only these perturbed eigenvalues, we count (with multiplicity) the number of real and negative such eigenvalues for $\la<\la_0$ and $\la>\la_0$ (in a neighbourhood of $\la_0$).  Since $0$ is an isolated eigenvalue, this number is constant on each of the intervals $(\la_0-\delta,\la_0)$, $(\la_0,\la_0+\delta)$ for some $\delta>0$ such that $\p_xF[\la,0]$ is regular on those intervals. Then the operator $\p_xF[\la,0]$ is said to have odd crossing number at $\la=\la_0$ if this number changes by an odd amount as $\la$ passes from $\la<\la_0$ to $\la>\la_0$.

 We are now is a position to state our main result.

\begin{thm} Let \label{mainthm}$h,k>0$ and $g,\gamma,\sigma\in\mathbb{R}$, $\si\ne0$, be fixed. For any $m\in\mathbb{R}$ there exist laminar flows with a flat free surface in water of depth $h$, with constant vorticity $\gamma$ and relative mass flux given by $m$. Away from the critical values for $m$ given by \begin{equation}m_\pm(n)=\frac{h^2\gamma}{2}-\frac{h\gamma\tanh(nkh)}{2kn}\pm h\sqrt{\frac{\gamma^2\tanh^2(nkh)}{4k^2n^2}+\frac{k^2n^2\sigma+g}{kn}\tanh(nkh)}.\label{ladef}\end{equation}all flows close to the laminar flow with relative mass flux $m$ are also laminar.

Let $m^*\in\{m_\pm(n):n\in\mathbb{N}\}$. For each such $m^*$ there exists a space $X^*\subset C^{2,\al}_{2\pi,0,e}$ and a (unique up to reparameterisation) continuous curve in $\R\times X^*$, given by $$\scr{C}_{m^*}:=\{(m_s,w_s):s\in\R\},$$of solutions to \eqref{2nd} such that $\Wkh(w_s)(t)\ne0$ for all $s,t\in\R$, and \begin{enumerate}\item[\emph{(i)}]$(m_0,w_0)=(m^*,0)$,\item[\emph{(ii)}]$w_s(t)=s\cos{nt}+o(s)$ as $s\to0$ for $0<|s|<\eps$ for some $\eps>0$ sufficiently small, where $n\in\mathbb{N}$ is such that $m^*=m_\pm(n)$ for some choice of $\pm$,\item[\emph{(iii)}]there is a neighbourhood of $(m^*,0)$ in $\R\times X^*$ and $\eps>0$ sufficiently small such that all solutions to \eqref{2nd} in that neighbourhood with $w\ne0$ are given by $(m_s,w_s)$ for $0<|s|<\eps$,\item[\emph{(iv)}]around each point of $\scr{C}_{m^*}$ there is a local real-analytic reparameterisation,\item[\emph{(v)}]one of the following occurs: either \begin{enumerate}\item we have that $$\min\left\{\frac{1}{1+\|m_s,w_s\|_{\R\times C^{2,\al}_{2\pi,0,e}}},\min_{t\in\R}\Wkh(w_s)(t)\right\}\to0\quad\text{as }s\to\pm\infty;$$or\item there exists $T>0$ such that $(m_{s+T},w_{s+T})=(m_s,w_s)$ for every $s\in\R$.\end{enumerate}\end{enumerate} Moreover, let $\cal{C}_{m^*}$ be the connected component of the set $$\overline{\{(m,w)\in\R\times X^*:(m,w)\text{ satisfy \eqref{2nd},}w\not\equiv0\}}$$ that contains the point $(m^*,0)$ (and thus the curve $\scr{C}_{m^*}$). Either \begin{enumerate}\item[\emph{(I)}]$\cal{C}_{m^*}$ is unbounded, or\item[\emph{(II)}]there exists a sequence of points $(m_n,w_n)\in\cal{C}_{m^*}$ such that either\begin{enumerate}[(A)]\item we have that $$\lim_{n\to\infty}\min_{t\in\R}\Wkh(w_n)(t)=0;$$or\item there exists $m'\in\R$ with $m'\ne m^*$ such that $(m_n,w_n)\to(m',0)$ as $n\to\infty$.\end{enumerate}\end{enumerate}\end{thm}\begin{rem} The conclusions of this theorem only hold when the kernel of the associated linearised operator has dimension 1; this is the reason for restricting to the subspace $X^*$. In the case where there exists exactly one choice of $\pm$ and exactly one $n\in\mathbb{N}$ such that $m^*=m_\pm(n)$, then one may take $X^*=C^{2,\al}_{2\pi,0,e}$.  It should be noted that for many choices of the parameters we can take $X^*=C^{2,\al}_{2\pi,0,e}$ for every $m^*$ (for example, it was shown in \cite{M} that if $\frac{\sigma}{gh^2}>\frac{\gamma^2h}{6g}+\frac13+\frac{|\gamma|}{6g}\sqrt{\gamma^2h^2+4gh}$ then the kernel is one-dimensional at every bifurcation point).\end{rem}

\begin{proof}[Proof of Theorem \ref{mainthm}] We first note that, due to the equivalences of (all) the formulations, it is sufficient to consider solutions to \eqref{gbeqn}.

The first statement of the theorem is then simply a formal statement of the clear fact that $w=0$ is a solution for any choice of $\lambda\in\mathbb{R}$. We therefore prove only the non-existence of bifurcation and seek to verify the conditions of Theorem \ref{CSVthm} and subsequently Theorem \ref{Rabthm}, in order to prove the remainder of the theorem as stated.

To begin, write $F(\lambda,w)=w+f(\lambda,w)$.  It is clear that $F\in C^1(\mathbb{R}\times\mathcal{U},X)$ and that $F(\lambda,0)=0$ for every $\lambda\in\mathbb{R}$. In fact, it is simple to observe that $F:\R\times\cal{U}\to X$ is a real-analytic mapping. It follows from the Implicit Function Theorem that for every $\lambda$ such that $\partial_wF[\lambda,0]$ is an invertible operator, there exists a neighbourhood of $(\lambda,0)$ in $\mathbb{R}\times\mathcal{U}$ such that in that neighbourhood all solutions of $F(\lambda,w)=0$ are the trivial solutions, with $w=0$. We therefore must find these values for $\lambda$. In fact, it turns out to be simpler to find the values for which $\partial_wF[\lambda,0]$ is \textit{not} invertible, this being a discrete set, and concluding invertibility -- and thus no bifurcation -- for all other values.

First let us compute $\partial_wF[\lambda,0]$. We observe that $\partial_wF[\lambda,0]=I+\partial_wf[\lambda,0]$, where $I$ denotes the identity operator, and so the difficulty will be computing $\partial_wf[\lambda,0]$. Since $D^{-2}$ is a linear operator, it is easy to see from the definition of $f$ that $$\partial_wf[\lambda,0]=-D^{-2}\circ\partial_wK[\lambda,0]$$and so we concentrate on computing $\partial_wK[\lambda,0]$.

We have, after a straightforward calculation, that \begin{equation}\partial_wK[\lambda,0]w=\frac{1}{2\sigma k}\left(\frac2k(g-\gamma\lambda)w-2\lambda^2\Ckh(w')\right)=\frac{g-\gamma\lambda}{\sigma k^2}w-\frac{\lambda^2}{\sigma k}\Ckh(w').\label{Kdiff}\end{equation}From this we thus conclude that \begin{equation}\partial_wf[\lambda,0]w=\frac{\gamma\lambda-g}{\sigma k^2}D^{-2}w+\frac{\lambda^2}{\sigma k}D^{-2}\Ckh Dw.\label{fdiff}\end{equation}Recalling (cf. \cite{CV}) that $\Ckh$ commutes with $D$, in view of \eqref{fdiff} we obtain \begin{equation}\partial_wF[\lambda,0]=I+\frac{\gamma\lambda-g}{\sigma k^2}D^{-2}+\frac{\lambda^2}{\sigma k}\Ckh D^{-1}.\label{Fdiff}\end{equation}Using a Fourier expansion to write $$w(t)=\sum_{n=1}^\infty a_n\cos{nt}$$and applying the following representations\begin{align*}D^{-1}:&\cos{nt}\mapsto\frac1n\sin{nt}&\Ckh:&\cos{nt}\mapsto\coth(nkh)\sin{nt}\\D^{-1}:&\sin{nt}\mapsto-\frac1n\cos{nt}&\Ckh:&\sin{nt}\mapsto-\coth(nkh)\cos{nt}\end{align*}for the linear operators $D^{-1},\Ckh$, we arrive from \eqref{Fdiff} at \begin{equation}\partial_wF[\lambda,0]w=\sum_{n=1}^\infty\left(1+\frac{g-\gamma\lambda}{k^2n^2\sigma}-\frac{\lambda^2\coth(nkh)}{kn\sigma}\right)a_n\cos{nt}.\label{FdiffF}\end{equation}From \eqref{FdiffF} it is clear that $\partial_wF[\lambda,0]$ is not invertible whenever \begin{equation}k^2n^2\sigma+g-\gamma\lambda-\lambda^2kn\coth(nkh)=0\label{pbpt}\end{equation}for some $n\in\mathbb{N}$. Observe at this point, that \eqref{pbpt} is identical to the condition derived in \cite{M}, as we would expect due to the equivalence between \eqref{next} and \eqref{gbeqn}. Solving \eqref{pbpt} for $\lambda$ we obtain \begin{equation}\lambda_\pm(n)=-\frac{\gamma\tanh(nkh)}{2kn}\pm\sqrt{\frac{\gamma^2\tanh^2(nkh)}{4k^2n^2}+\frac{k^2n^2\sigma+g}{kn}\tanh(nkh)}.\label{lavals}\end{equation}In view of \eqref{mtola}, we see that these values of $\lambda$ correspond exactly to the values for $m_\pm(n)$ given by \eqref{ladef}.

As remarked previously, for $\lambda\ne\lambda_\pm(n)$ for every $n\in\mathbb{N}$ we have via the Implicit Function Theorem that all solutions in a neighbourhood of $(\lambda,0)$ are those on the trivial curve, and the statement in the theorem follows.

Let $m^*$ be as given in the theorem and write $\la^*=m^*/h-h\gamma/2$. The argument presented in \cite{M} gives that there exists at most two distinct natural numbers $n_1,n_2\in\mathbb{N}$ such that $$m^*=m_\pm(n_1)=m_\pm(n_2),$$for some choice of $\pm$ in each case. If there exists exactly one $n\in\mathbb{N}$ such that $m^*=m_\pm(n)$ for some $\pm$, then we may set $X^*:=X$. This corresponds to some $\la^*\in\{\la_\pm(n):n\in\mathbb{N}\}$ such that there exists exactly one choice of $\pm$ and exactly one $n\in\mathbb{N}$ giving $\la^*=\la_\pm(n)$. However, if there do exist two distinct $n_1,n_2\in\mathbb{N}$, set $n=\max\{n_1,n_2\}$ and $$X^*:=X\cap\spn\{t\mapsto\cos{jnt}:j\in\mathbb{N}\}.$$In view of \eqref{pbpt} and \eqref{lavals}, it then follows that, when restricted to $X^*$, $\ker\p_wF[\la^*,0]=\spn\{t\mapsto\cos{nt}\}$ and therefore that $\dim\ker\p_wF[\la^*,0]=1$.

Setting $\cal{U}^*=\cal{U}\cap X^*$ and $\mathcal{O}=\mathbb{R}\times\mathcal{U}^*$, it should be clear that the final part of this theorem will follow from Theorems \ref{CSVthm} and \ref{Rabthm}, provided we can verify all of the requirements of these theorems.

We first prove that the mapping $f:\R\times\cal{U}\to X$ is compact. Indeed, observe that, as can be seen from its definition, we have $K:\R\times\cal{U}\to C^{1,\alpha}_{2\pi,0,e}$ and also that $D^{-2}:C^{p,\alpha}_{2\pi,0,e}\to C^{p+2,\alpha}_{2\pi,0,e}$ for $p\ge0$. Thus we obtain \begin{equation}f=-D^{-2}\circ K:\R\times\mathcal{U}\to C^{3,\alpha}_{2\pi,0,e}.\label{extrareg}\end{equation}But we only demand $f$ to have values in $X=C^{2,\alpha}_{2\pi,0,e}$. Recalling that $C^{q,\beta}$ is compactly embedded in $C^{p,\alpha}$ for any $q>p$ or $\beta>\alpha$ we see that $\mathrm{im}f$ is a compact subset of $X$ and so the mapping $f:\R\times\cal{U}\to X$ is indeed compact.

Also, since $f$ is a compact mapping, it follows that $\p_wf[\la,w]$ is a compact operator and that $\p_wF[\la,w]=I+\p_wf[\la,w]$ is a compact perturbation of the identity. Then, by dint of the Fredholm alternative \cite{BufTol}, we have that $\p_wF[\la,w]$ is a Fredholm operator of index $0$.

It is clear that the satisfaction of assumptions \emph{a.} and \emph{c.} of Theorem \ref{CSVthm} has been demonstrated. Since we have already established that $\dim\ker\p_wF[\la^*,0]=1$ and that the Fredholm property holds (giving $\codim\im\p_wF[\la^*,0]=1$ as required), to verify \emph{b.} it simply remains to show that $$\p^2_{\la w}F[\la^*,0](1,w^*)\notin\im\p_wF[\la^*,0],$$where $w^*(t)=\cos{nt}$. However, we can see from \eqref{FdiffF} that $$\p^2_{\la,w}[\la_0,0](1,w^*)=\left(\frac{-\gamma}{k^2n^2\sigma}-\frac{2\la^*\coth(nkh)}{kn\si}\right)\cos{nt},$$which will not be an element of $\im\p_wF[\la^*,0]$ provided that \begin{equation}\la^*\ne\frac{-\gamma\tanh(nkh)}{2kn};\label{ndr}\end{equation}or, equivalently, that $\la_+(n)\ne\la_-(n)$, i.e. $\la^*$ is not a double root of \eqref{pbpt}.  This is a mild restriction -- under the assumption\footnote{It can be seen that this assumption is not unreasonable, since $g$ and $\si$ are prescribed constants of gravitational acceleration and surface tension, respectively.} that $g,\si>0$ we have $\la_-(n)<0<\la_+(n)$ for every $n\in\mathbb{N}$ -- and so we assume that \eqref{ndr} holds.

Now let us prove that bounded, closed subsets of $\{(\la,w)\in\R\times\cal{U}:F(\la,w)=0\}$ are compact. Let $S\subset\{(\la,w)\in\R\times\cal{U}:F(\la,w)=0\}$ be bounded and closed. Define the sets $$S_1:=\{\la\in\R:(\la,w)\in S\text{ for some }w\in\cal{U}\},\qquad S_2:=\{w\in\cal{U}:(\la,w)\in S\text{ for some }\la\in\R\},$$which are readily seen to be themselves closed and bounded. As a closed bounded subset of $\R$, $S_1$ is compact. Since $f$ is a compact mapping and $S$ bounded, it follows that $\{-f(\la,w):(\la,w)\in S\}$ is a relatively compact set. But for $(\la,w)\in S$, we have that $-f(\la,w)=w$ and so $\{-f(\la,w):(\la,w)\in S\}$ is in fact exactly $S_2$. Since $S_2$ is closed it is therefore compact. As a closed bounded subset of the compact set $S_1\times S_2$, it follows that $S$ is compact.

Also, assumption \emph{d.} of Theorem \ref{CSVthm} is corroborated: we see that we may take any sequence $(\cal{Q}_j)_{j\in\mathbb{N}}$ of closed bounded sets, but in particular we may take \begin{equation}\cal{Q}_j=\left\{(\la,w)\in\cal{O}:\|\la,w\|_{\R\times X}\le j,\;\min_{t\in\R}\Wkh(w)(t)\ge\frac1j\right\},\label{calQdef}\end{equation}which can be trivially seen to satisfy the necessary conditions.

Hence, we may apply Theorem \ref{CSVthm} to obtain the curve $\scr{C}_{m^*}$, with each of its properties (i)-(v) corresponding to the properties \emph{1.}-\emph{5.} in Theorem \ref{CSVthm}. We obtain the property (v)\emph{(a)} by using the particular form for $\cal{Q}_j$ given by \eqref{calQdef}.

 We now ratify the assumptions of Theorem \ref{Rabthm}, in order to obtain the final part of this theorem. We have already shown that $f$ is compact, and thus completely continuous; it therefore only remains to prove that $\p_wF[\la,0]:X^*\to X^*$ has odd crossing number at the point $\la^*$ where, recalling our careful choice for $X^*$, we have that $$\dim\ker\p_wF[\la^*,0]=1.$$From the start of the proof we have that $\ker\partial_wF[\lambda^*,0]$ is generated by $w^*(t):=\cos{nt}$ and also that $w^*\notin\mathrm{im}\,\partial_wF[\lambda^*,0]$. Therefore $0$ is a simple eigenvalue of $\partial_wF[\lambda^*,0]$. Starting from the definitions in \cite{K}, it is clear that the linear operator $\partial_wF[\lambda,0]$ has odd crossing number at $\lambda=\lambda^*$ if the number of negative eigenvalues (counted with multiplicities) of $\partial_wF[\lambda,0]$ that tend to $0$ as $\la\to\la^*$ changes by an odd number as $\lambda$ passes from the interval $(\lambda^*-\delta,\lambda^*)$ to the interval $(\lambda^*,\lambda^*+\delta)$, for some $\delta>0$ such this number is constant on each interval. In our situation, since $0$ is a simple eigenvalue we have that there is exactly one such eigenvalue (with multiplicity 1) for any $\lambda$ sufficiently near $\lambda^*$. Therefore, to have an odd crossing number at $\lambda^*$ we need this eigenvalue to change sign as $\lambda$ passes from $\lambda<\lambda^*$ to $\lambda>\lambda^*$. It suffices to calculate the value of this eigenvalue of $\partial_wF[\lambda,0]$ and show that this changes sign. However, we can see from \eqref{FdiffF} that $\partial_wF[\lambda,0]$ has $w^*$ as an eigenvector, with eigenvalue $$1+\frac{g-\gamma\lambda}{k^2n^2\sigma}-\frac{\lambda^2\coth(nkh)}{kn\sigma},$$which can trivially be seen to converge to $0$ as $\la\to\la^*$. Hence, we have that this is the (only) eigenvalue of $\partial_wF[\lambda,0]$ in which we are interested, and it is non-zero for $\lambda\ne\lambda^*$ near $\lambda^*$. But we see that the function $$\Lambda(\lambda):=1+\frac{g-\gamma\lambda}{k^2n^2\sigma}-\frac{\lambda^2\coth(nkh)}{kn\sigma}$$is quadratic in $\lambda$ and $\Lambda(\lambda^*)=0$. Moreover, since $g,\si>0$, we observe that it follows from \eqref{lavals} that $\lambda_-(n)<0<\lambda_+(n)$ and that $\lambda_\pm(n)$ are the roots of $\Lambda$. Because any quadratic with distinct roots changes sign at each root we have that our eigenvalue also changes sign as $\lambda$ crosses $\lambda^*$, and hence the operator $\partial_wF[\lambda,0]$ has odd crossing number at $\lambda=\lambda^*$. We thus obtain that that the continuum $\cal{C}_{m^*}$ possesses all of the properties assured by Theorem \ref{Rabthm}, with the properties (I), (II)\emph{(A)}, (II)\emph{(B)} of $\cal{C}_{m^*}$ corresponding to the properties $(i),\;(ii),\;(iii)$ of Theorem \ref{Rabthm} respectively. This completes the proof.\end{proof}
\begin{rem}As remarked earlier, it was proved in \cite{M} that there do indeed exist values of the parameters giving a two-dimensional kernel, but that two is the largest possible dimension. Unfortunately, our result cannot be applied in the full space $X$ in the case of a two-dimensional kernel, since we end up studying two simple eigenvalues of the same sign: whilst these eigenvalues do change sign at $\lambda=\lambda^*$, the operator does not have an odd crossing number. As described in Theorem \ref{mainthm}, this issue can be circumvented by restricting to a suitable subspace. If, in the notation of the proof, it should occur that \begin{equation}\frac{\max\{n_1,n_2\}}{\min\{n_1,n_2\}}\notin\mathbb{N}\label{nores}\end{equation}then we may redefine $n=\min\{n_1,n_2\}$ and the proof proceeds as before with no alterations. In particular, this yields two (locally-real-analytically parameterisable) curves of solutions, one in the space of functions with minimum period $2\pi/n_1$ and one in the space of functions with minimum period $2\pi/n_2$. However, should condition \eqref{nores} fail, then we are unable to carry out the analysis to obtain this second curve.\end{rem}

\end{document}